\documentclass[11pt]{amsart}

\usepackage{amsfonts,amssymb,amsthm,amsmath,amscd,enumerate,tikz,tikz-cd,float,graphicx,overpic,url,comment,hyperref}
\usepackage[capitalise,noabbrev]{cleveref}
\usetikzlibrary{calc,decorations.pathreplacing,arrows}

\graphicspath{./drawings/}

\usepackage[backend=bibtex,
maxbibnames=5,
style=alphabetic,
giveninits=true,
doi=false,
url=false,
isbn=false]{biblatex}
\renewbibmacro{in:}{%
  \ifentrytype{article}
    {}
    {\printtext{\bibstring{in}\intitlepunct}}%
}

\DeclareFieldFormat[article]{volume}{\mkbibbold{#1}}
\DeclareFieldFormat[article]{number}{\mkbibbold{#1}}

\hypersetup{
    colorlinks = false,
    citecolor=green, 
    linkcolor=red,  
    urlcolor=cyan   
}

\numberwithin{equation}{section}
\newtheorem{theorem}{Theorem}
\newtheorem{lemma}[theorem]{Lemma}
\newtheorem{corollary}[theorem]{Corollary}
\newtheorem{proposition}[theorem]{Proposition}
\newtheorem{definition}[theorem]{Definition}
\theoremstyle{definition}

\newtheorem{remark}[theorem]{Remark}
\newtheorem{example}[theorem]{Example}
\numberwithin{theorem}{section}

\newcommand{\Z}{\mathbf{Z}}

\newcommand{\R}{\mathbf{R}}
\newcommand{\C}{\mathbf{C}}

\newcommand{\sepdim}{\mathrm{sepdim}}
\newcommand{\cone}{\mathrm{Cone}_{\omega}}

\title{Coarse median property of virtually nilpotent groups}
\author{Hyeonggeun Kim}
\address{Department of Mathematics, University of Michigan, Ann Arbor, MI 48109}
\email{kimhg@umich.edu}

\begin{document}
\maketitle

\section{Introduction}

The concept of a coarse median space introduced by Bowditch \cite{Bowditch2013CoarseMS} simultaneously generalizes Gromov-hyperbolic spaces and CAT(0) cube complexes. Coarse median spaces capture many of the interesting features of nonpositively curved spaces and include as an example the mapping class groups of surfaces. More generally, all hierarchically hyperbolic spaces  admit a coarse median structure \cite{Behrstock2019}. The existence of such a structure can be understood as a coarse nonpositive curvature condition. 

In this note, we show that virtually nilpotent groups are coarse median if and only if they are virtually abelian. The main idea is that the sub-Riemannian geometry of the asymptotic cone obstructs the existence of a locally convex Lipschitz median of finite rank. As an application, we deduce that non-compact lattices in the isometry group of a rank 1 symmetric space of non-compact type other than real hyperbolic space are not coarse median. This establishes the remaining case in the classification of lattices with the coarse median property initiated by Haettel \cite{Haettel2016}. The same approach applies more generally to complete finite-volume non-compact Riemannian manifolds $M$ of pinched negative sectional curvature: if at least one cusp cross-section does not admit a flat metric, then $\pi_1(M)$ is not coarse median.


\section{Background}
\subsection{Median spaces}
\begin{definition}
    Let $X$ be a set and $\mu:X^3 \to X$ be a ternary operation. The map $\mu$ is a \textbf{median} and $(X,\mu)$ is a \textbf{median algebra} if
\begin{itemize}
    \item $\mu(a,a,b) = a$,\ $\mu(a,b,c) = \mu(b,a,c) = \mu(a,c,b)$;
    \item $\mu(a,b,\mu(c,d,e)) = \mu(\mu(a,b,c),\mu(a,b,d),e)$.
\end{itemize}
The second axiom is referred to as the \textbf{five-point condition} of medians.
\end{definition}
\begin{example}
    The set $\left\{0,1\right\}$ has a unique median. The $n$-cube $\left\{0,1\right\}^n$ also has a unique median obtained by taking the product of the previous median.
\end{example}
\begin{definition}
    For median algebras $(X,\mu)$ and $(X',\mu')$, a map $f:X\to X'$ is \textbf{median} if $\mu'(f(x),f(y),f(z)) = f(\mu(x,y,z))$ for all $x,y,z\in X$. Moreover if $f$ is injective, it is called a \textbf{median embedding}. The \textbf{rank} of $X$ is the largest $n\geq 1$ for which there is a median embedding $\left\{0,1\right\}^n\to X$.
\end{definition}
\begin{definition}
    The \textbf{interval} between $a,b\in X$ is defined by
    $$ I(a,b) = \left\{c\in X: \mu(a,b,c) = c\right\}.$$
    A subset $C\subset X$ is called \textbf{convex} with respect to $\mu$ if $I(a,b)\subset C$ for all $a,b\in C$. If $X$ is additionally equipped with a metric, a subset $C\subset X$ is called $k$-\textbf{quasiconvex} with respect to $\mu$ if $I(a,b)\subset N_k(C)$ for all $a,b\in C$.
\end{definition}

To distinguish with the usual metric notions of convexity, we will always specify when subspaces are convex with respect to a median.

\begin{definition}
    Let $(X,d)$ be a metric space. The median $\mu$ on $X$ is called
    \begin{itemize}
        \item \textbf{continuous} if it is continuous as a map $X^3 \to X$;
        \item \textbf{Lipschitz} if is Lipschitz with respect to all three variables;
        \item \textbf{locally convex} if each point $X$ has a neighborhood basis of convex subsets.
    \end{itemize}    
\end{definition}
We will implicitly assume that all medians on metric spaces are continuous. The tuple $(X,d,\mu)$ is called a \textbf{metric median algebra}.

\begin{example}
     $(\R, L^1)$ is equipped with the `statistics median' which is Lipschitz and locally convex. Similarly, $(\R^n, L^1)$ admits a Lipschitz locally convex median by taking the component-wise statistics median.
\end{example}

\begin{definition}
    Let $X$ be a Hausdorff topological space. The \textbf{separation dimension} of $X$, which we denote as $\sepdim(X)$, is defined inductively as follows:
    \begin{itemize}
        \item $\sepdim(\emptyset) = -1$;
        \item $\sepdim(X)\leq n$ if for any distinct points $x,y\in X$, there are closed subsets $A,B\subset X$ such that $x\not\in A$, $y\not\in B$, $X = A\cup B$, and $\sepdim(A\cap B)\leq n-1$.
    \end{itemize}
\end{definition}
Note that the notion of separation dimension is purely topological. Some useful properties of separation dimension are the following:
\begin{lemma}
    Let $Y$ be a Hausdorff topological space and $X\subset Y$ be given the subspace topology. Then $\sepdim(X)\leq \sepdim(Y)$.\label{sepdim-order}
\end{lemma}
\begin{proof}
    We induct on $\sepdim(Y)$. Take distinct points $x,y\in X$. Take closed subsets $A,B\subset Y$ that separate $x,y$ and have $\sepdim(A\cap B)< \sepdim(Y)$. By the induction hypothesis, we have $\sepdim(A\cap B\cap X)\leq \sepdim(A\cap B)$. Since $A\cap X$ and $B\cap X$ are closed subsets of $X$ that separate $x$ and $y$, we deduce that $\sepdim(X)\leq \sepdim(Y)$.
\end{proof}
For nice enough spaces, separation dimension coincides with the more classical notion of (large) inductive dimension.
\begin{definition}
    Let $X$ be a normal topological space. The \textbf{inductive dimension} of $X$, which we denote $\mathrm{Ind}(X)$, is defined inductively as follows:
    \begin{itemize}
        \item $\mathrm{Ind}(\emptyset) = -1$;
        \item $\mathrm{Ind}(X) \leq n$ if for any closed $A\subset X$ and open $V\subset X$ containing $A$, there is an open subset $U\subset X$ such that $A\subset \bar{U}\subset V$ and $\mathrm{Ind}(\bar{U} \setminus U) \leq n-1$.
    \end{itemize}
\end{definition}
\begin{lemma}[{\cite[Section III.6]{HurewiczWallman}}]
    If $X$ is a locally compact Hausdorff metric space, then $\sepdim(X) = \mathrm{Ind}(X)$.\label{ind-equal-sep}
\end{lemma}
By the lemma above, we can deduce that the separation dimension of $\R^n$ is equal to $n$. The separation dimension precisely detects the rank of locally convex Lipschitz medians on the space.
\begin{theorem}[{\cite[Theorem 2.2]{Bowditch2013CoarseMS}}]
    Any topological space with a locally convex median of rank $r$ has separation dimension at most $r$.\label{bowditch-upper-bound}
\end{theorem}
\begin{theorem}[{\cite[Proposition G]{Haettel2016}}]
Let $(X,d)$ be a connected metric space equipped with a locally convex Lipschitz median $\mu$ of rank $r$. Then there exists a bi-Lipschitz median embedding of the $r$-cube $[0,1]^r$ into $X$ with convex image with respect to $\mu$.\label{haettel-cube-embedding}
\end{theorem}
\begin{corollary}
    Let $X$ be a connected metric space with a Lipschitz locally convex median of rank $r$. Then $\sepdim(X) = r$.\label{sepdim-equals-rank}
\end{corollary}
\begin{proof} 
    Take the median embedding $[0,1]^r\to X$. Since this map is a homeomorphism onto its image, we have that $\sepdim(X) \geq r$ by \cref{sepdim-order}. But we also have $\sepdim(X)\leq r$ by the theorem above.
\end{proof}
\begin{example}
    The statistics median on $(\R^n, L^1)$ has rank $n$.
\end{example}
\begin{remark}
    Note that if a metric space $X$ admits a median, then so does any metric space $Y$ bi-Lipschitz to $X$. The bi-Lipschitz and local convexity properties of medians are also preserved under bi-Lipschitz equivalence of the base space.
\end{remark}

\subsection{Coarse medians and relative hyperbolicity}
\begin{definition}
    A ternary map $\mu:X^3\to X$ is called a \textbf{coarse median} on $X$ if there is a constant $k>0$ and a function $h:\Z_{\geq 0} \to [0,\infty)$ such that:
    \begin{itemize}
        \item $d(\mu(x,y,z),\mu(x',y',z')) \leq k(d(x,x') + d(y,y') + d(z,z')) + h(0)$;
        \item for any $n\geq 1$ and finite subset $A\subset X$ of size at most $n$, there is a finite median algebra ($\Pi,\mu_{\Pi})$ and maps $\pi:A\to \Pi$, $\lambda:\Pi\to X$ such that
        $$ d(\lambda \mu_{\Pi}(p,q,r),\mu(\lambda p,\lambda q,\lambda r))\leq h(n)$$
        for all $p,q,r\in \Pi$ and $d(a,\lambda \pi a)\leq h(n)$ for all $a\in A$.
    \end{itemize}
    The space $(X,d,\mu)$ is called a \textbf{coarse median space}. If $\Pi$ can always be taken to have rank at most $r$, we say that $\mu$ has \textbf{rank at most $r$}.
\end{definition}
If $X$ admits a coarse median of finite rank, we will often simply say that $X$ is coarse median. We say that a finitely generated group $G$ is coarse median if any, and hence all, of its Cayley graphs admit a coarse median of finite rank. (Quasi)convexity of subspaces with respect to coarse medians are similarly defined. We will only consider coarse medians of finite rank in this note.

Coarse median spaces can be interpreted as coarse analogues of metric median algebras, or more accurately, spaces in which finite subsets are uniformly well-approximated by finite median algebras. One motivation for considering such spaces is that their asymptotic cones are equipped with metric medians with nice properties.
\begin{definition}
    Take a basepoint $e\in X$, a nonprincipal ultrafilter $\omega$ on $\Z_{>0}$, and a sequence of positive real numbers $(\lambda_i)_{i=1}^{\infty}$ such that $\lim_{\omega}\lambda_i = \infty$. Consider the following space equipped with the pseudometric $d_{\infty}$:
    $$ \left\{(x_i)_{i=1}^{\infty}: d(x_i,e)/\lambda_i \text{ is } \omega\text{-bounded}\right\},\quad d_{\infty}((x_i),(y_i)) = \lim_\omega d(x_i,y_i)/\lambda_i.$$
    The \textbf{asymptotic cone} of $X$, denoted $\cone(X)$, is the metric space obtained by identifying pairs of points of zero distance in the space above.
\end{definition}
Note that the resulting space does not depend on the choice of basepoint $e\in X$, but does depend on the scaling factors $(\lambda_i)$ and the ultrafilter $\omega$. 
\begin{theorem}[{\cite[Theorem 2.3]{Bowditch2013CoarseMS}}]
    Let $(X,d,\mu)$ be coarse median of rank at most $r$. Then any asymptotic cone of $X$ admits a locally convex Lipschitz median of rank at most $r$.\label{cone-admits-lipschitz-median}
\end{theorem}

Another nice property of coarse medians is that they behave well with relative hyperbolicity.
\begin{theorem}
Suppose group $G$ is hyperbolic relative to finitely generated subgroups $H_1,\ldots, H_m$. Then $G$ is coarse median of rank at most $r$ if and only if each $H_i$ is coarse median of rank at most $r$.\label{iff-relhyp}
\end{theorem}
\begin{proof}
    The `if' direction is proved by Bowditch \cite{Bowditch2013RelHyp}. For the `only if' direction, suppose that $G$ admits a coarse median $\mu$ of rank $\leq r$. Note that each peripheral subgroup $H_i$ is strongly quasiconvex in $G$ by \cite[Lemma 4.15]{Drutu2004TreegradedSA}. 
    It follows from \cite[Corollary 3.4]{FioravantiSisto2025} that each $H_i$ is quasiconvex with respect to $\mu$. Hence we obtain rank $r$ coarse medians on the peripherals by restricting and taking closest point projections onto each subgroup.
\end{proof}

\section{Main results and proofs}

\begin{proposition}
Let $N$ be a finitely generated group virtually nilpotent which is not virtually abelian. Then no asymptotic cone of $N$ admits a locally convex Lipschitz median of finite rank. \label{nilpotent-median}
\end{proposition}
\begin{proof}
    It is known that $N'=\cone(N)$ is a non-abelian nilpotent Lie group with a metric bi-Lipschitz equivalent to the Carnot-Carath\`{e}odory metric \cite{Pansu1983CroissanceDB}. Assume for contradiction that $N'$ admits a locally convex Lipschitz median of finite rank. Let $r$ be the dimension of $N'$ as a manifold. By the Chow-Rashevskii theorem, $N'$ with the topology induced from its metric is homeomorphic to its intrinsic, locally Eulidean topology \cite{Chow1939}. Since $N'$ is also locally compact, it follows from \cref{ind-equal-sep} that $N'$ has separation dimension $r$. Hence the median on $N'$ has rank $r$ by \cref{sepdim-equals-rank}.
    
    By \cref{haettel-cube-embedding}, there is a bi-Lipschitz embedding $\phi:(0,1)^r \to N'$. By invariance of domain, the image of $\phi$ is an open subset. Since $N'$ is covered by countably many left translates of this open subset, its Hausdorff dimension must equal $r$. This is a contradiction since $r$-dimensional non-abelian nilpotent Lie groups with the Carnot-Carath\`{e}odory metric have Hausdorff dimension greater than $r$ \cite{Mitchell1985}, and Hausdorff dimension is invariant under bi-Lipschitz equivalence.
\end{proof}

\begin{corollary}
    Let $G$ be a finitely generated virtually nilpotent group. Then $G$ is coarse median if and only if it is virtually abelian. \label{main}
\end{corollary}
\begin{proof}
    If $G$ is virtually nilpotent and coarse median, its asymptotic cone admits a locally convex Lipschitz median of finite rank by \cref{cone-admits-lipschitz-median}. Hence by \cref{nilpotent-median}, $G$ must be virtually abelian. The converse follows from $\mathbb{Z}^r$ being a coarse median space.
\end{proof}

\begin{corollary}
Let $G$ be a relatively hyperbolic group with virtually nilpotent peripherals. Then $G$ is coarse median if and only if all peripherals are virtually abelian.\label{cm-iff}
\end{corollary}
\begin{proof}
This follows from \cref{iff-relhyp} and \cref{main}.
\end{proof}

In \cite{Haettel2016}, Haettel proved that thick affine buildings of spherical type other than $A_1^r$ cannot admit a locally convex Lipschitz median as there is `too much branching' of flats. This observation was generalized to the concept of \textit{richly branching flats} in \cite{MunroPetyt2025}. The analysis of \cref{nilpotent-median} shows that the sub-Riemannian geometry of nilpotent Lie groups provides another obstruction to the existence of such medians.

By passing to the asymptotic cone, Haettel consequently showed that all lattices in higher rank simple algebraic groups over arbitrary local fields are not coarse median. This is in contrast to co-compact rank 1 lattices and non-compact lattices in $SO_0(n,1)$ which are coarse median. We analogously show that non-compact lattices in any other isometry group of rank 1 symmetric space cannot be coarse median. This follows naturally by viewing such lattices as relatively hyperbolic groups with virtually nilpotent peripherals.

\begin{corollary}
Let $\Gamma$ be a non-compact lattice in the isometry group of a rank 1 symmetric space of non-compact type other than real hyperbolic space. Then $\Gamma$ is not coarse median. \label{cor-mfds}
\end{corollary}
\begin{proof}
The lattice $\Gamma$ is a relatively hyperbolic group whose peripherals are modelled on generalized Heisenberg groups, which are virtually nilpotent but not virtually abelian \cite{Farb1998}. The claim follows from \cref{cm-iff}.
\end{proof}
\begin{example}
Recall that complex hyperbolic $n$-space has isometry group $PU(n,1)\rtimes \Z/2$. Explicit examples of $\Gamma$ as above are thus given by
$$PU(n,1;\mathcal{O}_K) = P(SU(n,1)\cap SL(n+1,\mathcal O_K))$$
for $n\geq 2$, where $K$ is an imaginary quadratic field and $\mathcal O_K$ is its ring of integers \cite{Morris2015}. For $n=2$, these are the classical \textit{Picard modular groups}.
\end{example}

The results above lead to the classification of the coarse median property for lattices in simple algebraic groups. All that remains to be considered are the rank 1 lattices over base fields other than $\R$: For $k = \C$ the only possible ambient groups are $SL(2,\C)$ and $PSL(2,\C)$, both of which are locally isomorphic to $SO_0(3,1)$. For non-Archimedean $k$, there are no finitely generated non-compact lattices \cite{Lubotzky1991}.
\begin{theorem}
    Let $k$ be a local field and $\Gamma$ be a lattice in the group $G$ of $k$-points of a simple algebraic group without compact factors.
    \begin{itemize}
        \item If $G$ is rank 1 and $\Gamma$ is co-compact, then $\Gamma$ is coarse median.
        \item If $G$ is locally isomorphic to $SO_{0}(n,1)$ and $\Gamma$ is non-compact, then $\Gamma$ is coarse median.
        \item If $G$ is rank 1 but not locally isomorphic to $SO_0(n,1)$ and $\Gamma$ is non-compact, then $\Gamma$ is not coarse median.
        \item If $G$ is higher rank, then $\Gamma$ is not coarse median.
    \end{itemize}
\end{theorem}

We conclude with a geometric formulation and extension of \cref{cor-mfds}. It shows that for a finite-volume non-compact locally symmetric manifold $M$ of negative curvature which is not locally isometric to real hyperbolic space, $\pi_1(M)$ is not coarse median. The same conclusion, with the same proof, applies more generally to a complete finite-volume non-compact Riemannian manifold $M$ of pinched negative sectional curvature with at least one cusp cross-section $Q$ that does not admit a flat metric. Note that since $Q$ is a closed infranilmanifold \cite{BGS1985}, $\pi_1(Q)$ is virtually abelian if and only if $Q$ admits a flat metric. Many non-locally symmetric examples of such $M$ are provided by Riemannian hyperbolization \cite{Ontaneda2020}.

\subsection*{Acknowledgements}
I would like to thank my advisor Alex Wright for generously sharing his time and insights with me. I am also grateful to Alessandro Sisto and Ralf Spatzier for the helpful comments and discussions.

\newpage

\printbibliography
\end{document}